\documentclass[preprint]{imsart}

\RequirePackage[OT1]{fontenc}
\usepackage{amsthm,amsmath,natbib}
\RequirePackage[colorlinks,citecolor=blue,urlcolor=blue]{hyperref}

\startlocaldefs
\numberwithin{equation}{section}
\theoremstyle{plain}
\newtheorem{thm}{Theorem}[section]
\newtheorem{lemma}{Lemma}[section]

\newtheorem{proposition}{Proposition}[section]
\theoremstyle{remark}
\newtheorem{remark}{Remark}[section]
\endlocaldefs

\begin{document}

\begin{frontmatter}
\title{Bayesian predictive densities for linear regression models
under $\alpha$-divergence loss: some results and open problems}
\runtitle{Bayesian predictive densities}

\begin{aug}
\author{\fnms{Yuzo} \snm{Maruyama}}
\and
\author{\fnms{William, E.} \snm{Strawderman}}

\address{The University of Tokyo and Rutgers University 
}

\runauthor{Y. Maruyama and W. Strawderman}

\affiliation{The University of Tokyo and Rutgers University}
\end{aug}

\begin{abstract}
This paper considers estimation of the predictive density 
for a normal linear model with unknown variance under $\alpha$-divergence loss 
for $ -1 \leq \alpha \leq 1$.
We first give a general canonical form for the problem, 
and then give general expressions for the generalized Bayes solution 
under the above loss for each $\alpha$. 
For a particular class of hierarchical generalized priors studied 
in Maruyama and Strawderman (2005, 2006) for the problems of estimating 
the mean vector and the variance respectively, 
we give the generalized Bayes predictive density. 
Additionally, we show that, for a subclass of these priors, 
the resulting estimator dominates the generalized Bayes estimator 
with respect to the right invariant prior when $\alpha=1$, i.e., 
the best (fully) equivariant minimax estimator.
\end{abstract}




\end{frontmatter}

\section{Introduction}
\label{sec:intro}
We begin with the standard normal linear regression model
setup
\begin{equation} \label{observation}
y \sim N_n(X\beta,\sigma^2I_n) ,
\end{equation}
where $y$ is an $n \times 1$ vector of  observations,
$X$ is an  $n \times k$ matrix of $k$ potential 
predictors where $n > k$ and rank $X=k$, 
and $\beta$ is a $k \times 1$ 
vector of unknown regression coefficients, and $\sigma^2$
is  unknown variance.
Based on observing $y$, we consider the problem of
giving the predictive density $p(\tilde{y}|\beta, \sigma^2)$
of a future $m \times 1$ vector $\tilde{y}$ 
where
\begin{equation}\label{future}
\tilde{y} \sim N_m(\tilde{X}\beta,\sigma^2I_m) .
\end{equation}
Here $ \tilde{X}$ is a fixed 
$m \times k$ design matrix of the same $k$ potential predictors in $X$,
and the rank of $\tilde{X}$ is assumed to be $\min(m,k)$.
We also assume that $y$ and $\tilde{y}$ are conditionally 
independent given $\beta$ and $\sigma^2$.
Note that in most earlier papers on such prediction problems, $\sigma^2$ 
is assumed known, partly because this typically makes the problem less difficult. 
However, the assumption of unknown variance is more realistic, 
and we treat this more difficult case in this paper. 
In the following we denote by $\psi$ all the unknown parameters $\{\beta,\sigma^2\}$.

For each value of $y$, a predictive estimate 
$\hat{p}(\tilde{y};y)$ of $p(\tilde{y}|\psi)$
is often evaluated by the Kullback-Leibler (KL) divergence
\begin{equation} 
\label{KL-loss}
D_{KL}\left\{\hat{p}(\tilde{y};y),  p(\tilde{y}|\psi)\right\}
=\int p(\tilde{y}|\psi) \log
 \frac{p(\tilde{y}|\psi)}{\hat{p}(\tilde{y};y)} d\tilde{y}
\end{equation}
which is called the KL divergence loss from 
$p(\tilde{y}|\psi)$ to $\hat{p}(\tilde{y};y)$.
The overall quality of the procedure $\hat{p}(\tilde{y};y)$
for each $\psi$ is then
conveniently summarized by the KL risk
\begin{equation}
\label{KL-risk}
 R_{KL}(\hat{p}(\tilde{y};y),  \psi)
= \int D_{KL}\left\{\hat{p}(\tilde{y};y),
 p(\tilde{y}|\psi)\right\} p(y|\psi)dy
\end{equation}
where $p(y|\psi)$ is the density of $y$ in \eqref{observation}.
\cite{Aitchison-1975} showed that 
the Bayesian solution with respect to
the prior $\pi(\psi)$ under the loss $D_{KL}$ given by \eqref{KL-loss}
is what is called
the Bayesian predictive density
\begin{equation} 
\label{Bayes-predictive}
\hat{p}_{\pi}(\tilde{y};y) 
=\frac{\int p(\tilde{y}|\psi) 
 p(y|\psi)\pi(\psi)d\psi} 
{\int  p(y|\psi)
\pi(\psi)d\psi} 
=\int p(\tilde{y}|\psi) \pi(\psi| y)d\psi 
\end{equation}
where 
\begin{equation*}
 \pi(\psi| y)
=\frac{p(y|\psi)\pi(\psi)}
{\int p(y|\psi)\pi(\psi) d\psi}.
\end{equation*}
For the prediction problems in general, 
many studies suggest the use of the Bayesian predictive
density rather than plug-in densities of the form
$p(\tilde{y}|\hat{\psi}(y))$, where $\hat{\psi}$ is an estimated
value of $\psi$. 
In our setup of the problem, 
\cite{Liang-Barron-2004} showed that the Bayesian predictive
density with respect to the right invariant prior 
is the best equivariant and minimax.
Although the Bayesian predictive density with respect to the right invariant
prior is a good default procedure, it has been shown to be inadmissible in some cases.
Specifically, when $\sigma^2$ is assumed to be known and the following are assumed
\begin{equation} \label{special-design-matrices} 
\begin{split}
m \geq k \geq 3, \ n=mN \\
X=1_N \otimes \tilde{X}=(\tilde{X}', \dots, \tilde{X}' )' 
\end{split}\tag{AS1} 
\end{equation}
where $N$ is an positive integer, $1_N$ 
is an $N \times 1$ vector each component of which is one,
and $\otimes $ is the Kronecker product,
\cite{Komaki-2001} showed that 
the shrinkage Bayesian predictive density with respect to the harmonic
prior 
\begin{equation}
\pi_{S,0}(\psi)=\pi(\beta)=\{\beta'\tilde{X}'\tilde{X}\beta\}^{1-k/2}
\end{equation}
dominates the best invariant Bayesian predictive density 
with respect to 
\begin{equation}
\pi_{I,0}(\psi)=\pi(\beta)=1.
\end{equation}
\cite{George-etal-2006} extended \cite{Komaki-2001}'s result to
general shrinkage priors including \cite{Strawderman-1971}'s  prior.
As pointed out in the above, we will assume that the variance $\sigma^2$
is unknown in this paper. 
The first decision-theoretic result in the unknown variance case
was derived by \cite{Kato-2009}. 
He showed that, under the same assumption of \cite{Komaki-2001} given by
\eqref{special-design-matrices},
the Bayesian predictive density 
with respect to the shrinkage prior
\begin{equation}\label{kato-prior}
\pi_{S,1}(\psi)=\pi(\beta,\sigma^2)
=\{\beta'\tilde{X}'\tilde{X}\beta\}^{1-k/2}\{\sigma^2\}^{-2}
\end{equation}
dominates the best invariant predictive density which is
the Bayesian predictive density 
with respect to the right invariant prior 
\begin{equation}
\label{right-invariant-prior-unknown}
\pi_{I,1}(\psi)=\pi(\beta,\sigma^2)=\{\sigma^{2}\}^{-1}.
\end{equation}

From a more general viewpoint,
the KL-loss given by \eqref{KL-loss}
is in the class of $\alpha$-divergence 
introduced by \cite{Csiszar-1967} and defined by
\begin{equation} \label{alpha-D}
 D_{\alpha}\{\hat{p}(\tilde{y};y),  p(\tilde{y}|\psi)\}
=\int f_{\alpha}\left( 
\frac{\hat{p}(\tilde{y};y)}{p(\tilde{y}|\psi)}
\right)
p(\tilde{y}|\psi)
d \tilde{y}
\end{equation}
where
\begin{equation*}
 f_{\alpha}(z)=
\begin{cases}
\frac{4}{1-\alpha^2}(1- z^{(1+\alpha)/2}) & \ |\alpha| <1 \\
z\log z & \ \alpha =1 \\
-\log z & \ \alpha = -1.
\end{cases}
\end{equation*}
Clearly the KL-loss given by \eqref{KL-loss} corresponds to $D_{-1}$.
\cite{Corcuera-Giummole-1999} showed that
a generalized Bayesian predictive density under $D_\alpha$
is
\begin{equation} \label{gene-B-density}
 \hat{p}_{\pi,\alpha}(\tilde{y};y) \propto
\begin{cases}
\left[\int p^{(1-\alpha)/2}(\tilde{y}|\psi)
\pi(\psi|y)d\psi
 \right]^{2/(1-\alpha)} & \alpha \neq 1 \\
 \exp\{ \int \log p(\tilde{y}|\psi)\pi(\psi|y)d\psi\} & \alpha=1.
\end{cases}
\end{equation}
Hence
the Bayesian predictive density of the form \eqref{Bayes-predictive}
may not be good under $\alpha$-divergence with $\alpha \neq -1$.
But as \cite{Brown-1979} pointed out in the estimation problem,
decision-theoretic properties often seem to depend on the general
structure of the problem (the general type of problem (location, scale), and
the dimension of the parameter space) and on the prior in a Bayesian-setup,
but not the loss function.
In fact, we will show, 
under the assumption \eqref{special-design-matrices} and the $D_{1}$ loss,
the predictive density with respect to the same shrinkage prior given by
\eqref{kato-prior} improves on the best invariant predictive density
with respect to \eqref{right-invariant-prior-unknown}
(See Section \ref{sec:minimax-D_1}). 
From this viewpoint, we are generally interested in how robust the Stein effect 
already founded under $D_\alpha$ loss for a specific $\alpha$
is under $D_\alpha$ loss for general  $\alpha$.
For example, we can find some concrete problems as follows.
\begin{description}
\item[Problem 1]
Under the assumption \eqref{special-design-matrices} and the $D_{\alpha}$ loss
for $-1<\alpha<1$,
does the predictive density with respect to the same shrinkage prior given by
\eqref{kato-prior} improve on the best invariant predictive density
with respect to \eqref{right-invariant-prior-unknown}? 
\item[Problem 2-1]
Under $D_1$ loss, even if $k=1,2$,
the best invariant predictive density remains inadmissible
because an improved non-Bayesian predictive density is easily found. 
(See Section \ref{sec:minimax-D_1}.)
Can we find improved Bayesian predictive densities for this case ($k=1,2$)?
\item[Problem 2-2]\label{prob2.2}
Under $k=1,2 $ and the $D_\alpha$ loss with $-1 \leq \alpha <1$, 
does the best invariant predictive density keep inadmissibility?
If so, which Bayesian predictive density improve it? 
\end{description}
In this paper, a main focus is on Problem 2-1 and 2-2.
For Problem 2-1, we will give an exact solution.
We could not solve Problem 2-2 in this paper, but by a natural extension of 
the shrinkage prior considered for Problem 2-1 ($D_1 $ loss),
we will provide a class of predictive densities which we hope 
lead the solution in the future work.
In addition, Problem 1 is open.

The organization of this paper is as follows.
We treat not only simple design matrices like \eqref{special-design-matrices}
but also general ones noted at the beginning of this section.
In order to make the structure of our problem clearer,
Section \ref{sec:canonical} gives the canonical form of the problem.
In Section \ref{sec:density}, we consider a natural extension of a
hierarchical prior which was originally 
proposed in \cite{Strawderman-1971} and \cite{Maru-Straw-2005} 
for the problem of estimating $\beta$. Using it,
we will construct a Bayesian predictive density under $D_\alpha$ loss 
for $-1 \leq \alpha < 1$ and $\alpha=1$. 
In Section \ref{sec:minimax-D_1}, we show that a subclass of
the Bayesian predictive densities proposed 
in Section \ref{sec:density} is minimax under $D_1$ loss 
even if $k$ is small.
Section \ref{sec:cr} gives concluding remarks.
\section{A canonical form}
\label{sec:canonical}
In the section, we reduce the problem to a canonical form.
To simplify expressions and to make matters a bit clearer it is helpful
to rotate the problem via the following transformation.
First we note that for the observation $y$,
sufficient statistics are
\begin{equation*}
\begin{split}
 \hat{\beta}_U &=(X'X)^{-1}X'y \sim N_k(\beta,\sigma^2(X'X)^{-1}), \\
 \quad S &=
\|(I- X(X'X)^{-1}X')y\|^2 \sim \sigma^2 \chi_{n-k}^2
\end{split}
\end{equation*}
where $ \hat{\beta}_U$ and $S$ are independent.

\smallskip 

{\bfseries Case I:} When $m \geq k$, 
let $M$ be a nonsingular $k \times k$ matrix which simultaneously
diagonalizes 
matrices $X'X $ and $\tilde{X}'\tilde{X}$, where 
$M$ satisfies
\begin{equation*}
M'(X'X)^{-1}M=\mbox{diag}(d_1, \dots , d_k), \quad  
MM'=\tilde{X}'\tilde{X}
\end{equation*}
where $d_1 \geq \dots \geq d_k$. Let $V=M'\hat{\beta}_U$ and $
\theta=M'\beta$.  

\smallskip

{\bfseries Case II:} When $m<k$, 
there exists an $(k-m)\times k$ matrix $\tilde{X}_*$ such that
$(\tilde{X}', \tilde{X}'_*)'$ is a $k\times k$ non-singular matrix
and also $\tilde{X}(X'X)^{-1}\tilde{X}'_*$ is an $m \times (k-m)$
zero matrix.
Further 
there exists an $m \times m$ orthogonal matrix $P$
which diagonalizes $ \sigma^2\tilde{X}(X'X)^{-1}\tilde{X}'$, 
the covariance matrix of $\tilde{X}\hat{\beta}_U$,
i.e.,
\begin{equation*}
 P'\tilde{X}(X'X)^{-1}\tilde{X}'P=\mbox{diag}(d_1,\dots ,d_m)
\end{equation*}
where $ d_1 \geq \dots \geq d_m$. 
There also exists a $(k-m)\times (k-m)$ matrix $P_*$ such that
\begin{equation*}
 P'_*\tilde{X}_*(X'X)^{-1}\tilde{X}'_* P_*=I_{k-l}.
\end{equation*}
Then $V$ and $V_*$ 
where
\begin{equation*}
\begin{pmatrix}
 V \\ V_* 
\end{pmatrix}
= \begin{pmatrix}
  P' & 0 \\
0 & P'_*
 \end{pmatrix}
\begin{pmatrix}
 \tilde{X} \\
\tilde{X}_*
\end{pmatrix}
\hat{\beta}_U 
\end{equation*}
are independent and have multivariate normal distributions
$N_{m}(P'\tilde{X}\beta,\sigma^2D)$ and
$N_{k-m}(P'_*\tilde{X}_*\beta,\sigma^2I_{k-m})$
respectively. Let $\theta=P'\tilde{X}\beta$ and
$\mu =P'_*\tilde{X}_*\beta$.

\smallskip

In summary, a 
canonical form of the prediction problem
is as follows.
We observe
\begin{equation} 
\label{canonical-observation}
 V \sim N_l(\theta,\eta^{-1} D), \quad V_* \sim N_{k-l}(\mu,\eta^{-1} I),
 \quad \eta S \sim  \chi_{n-k}^2
\end{equation}
where $\eta=\sigma^{-2}$, $l=\min(k,m)$, $D=\mbox{diag}(d_1,\dots,d_l)$ and $d_1 \geq
\dots \geq d_l$. When $m \geq k$, $V_*$ is empty.
Then the problem is  to give a predictive density of 
an $m$-dimensional future observation
\begin{equation} 
\label{canonical-tilde-y}
 \tilde{Y} \sim N_m(Q\theta, \eta^{-1}I_m)
\end{equation}
where $Q$ is an $m\times l$ matrix, which 
is given by
\begin{equation*}
 Q=\begin{cases}
    P & \mbox{if }m<k \\
\tilde{X}(M')^{-1} & \mbox{if }m \geq k,
   \end{cases}
\end{equation*}
and hence satisfies $Q'Q=I_l$.
Notice that, under the assumption given by\eqref{special-design-matrices},
$D$ becomes $N^{-1}I_k$, $V_*$ is empty, and
$Q$ becomes $\tilde{X}(\tilde{X}'\tilde{X})^{-1/2}$.

The distribution of $\tilde{y}$ in \eqref{canonical-tilde-y} 
is the same as in \eqref{future},
 so it is just the $\tilde{y}$'s that have been transformed.
In the remainder of the paper, we will consider the problem
in its canonical form, \eqref{canonical-observation} 
and \eqref{canonical-tilde-y}.
We will  use the notation $\hat{p}(\tilde{y}|y)$ in the following
although it may be more appropriate to use $\hat{p}(\tilde{y}|v,v_*,s)$ 
or $\hat{p}(\tilde{y}|\hat{\beta}_U,s)$.

\section{A class of generalized Bayes predictive densities}
\label{sec:density}
In this section,
we consider the following class of hierarchical prior densities,
$\pi(\theta,\mu,\eta)$, for 
the canonical model given by \eqref{canonical-observation} 
and \eqref{canonical-tilde-y}.
\begin{equation} 
\label{our-canonical-prior}
\begin{split}
\theta|\eta,\lambda &\sim N_l \left(0,\eta^{-1}
\left(D^{-1}+\{(1-\alpha)/2\}I_l \right)^{-1}
\left( C/\lambda -I_l \right)\right)  \\ 
\mu|\eta,\lambda &\sim N_{k-l}\left( 0,\eta^{-1} (\gamma/\lambda-1)
 I_{k-l}\right) \\
\eta & \propto \eta^a, \quad \lambda \propto
 \lambda^a(1-\lambda)^bI_{(0,1)}(\lambda), 
\end{split} 
\end{equation}
where $C=\mbox{diag}(c_1,\dots,c_l)$ with $c_i \geq 1$ for $1\leq i \leq
l$, $b=b(\alpha)=(1-\alpha)m/4+(n-k)/2-1$ 
and $\gamma \geq 1$. The integral which appears in 
the Bayesian predictive density below will be well-defined when
$a> -k/2-1$.
An essentially equivalent class was considered 
for the problem of estimating $\theta$ and $\sigma^2$ 
in \cite{Maru-Straw-2005, Maru-Straw-2006} respectively.
When $m \geq k$, the prior on $\mu$ is empty and 
we have only to eliminate $\|V_*\|^2/\gamma$ from the representation of
the Bayesian solution in the following theorems 
\ref{thm:case-1} and \ref{thm:case-2}, 
in order to have the corresponding result.

\subsection{Case i: $\alpha \in [-1, 1)$}
\label{sec:case-1}
\begin{thm} \label{thm:case-1}
 The generalized Bayes predictive density under
$D_\alpha$ divergence with respect to the prior \eqref{our-canonical-prior}
is given by
\begin{equation}
\hat{p}_\alpha(\tilde{y}|y) \propto
\hat{p}_{\{U,\alpha\}}(\tilde{y}|y) \times \check{p}_\alpha(\tilde{y}|y),
\end{equation}
where
\begin{equation}
\begin{split}
\hat{p}_{\{U,\alpha\}}(\tilde{y}|y) &=
\left\{ (\tilde{y}-Qv)'\Sigma_U^{-1}
(\tilde{y}-Qv) 
+ s \right\}^{-m/2-(n-k)/(1-\alpha) }, \\
\check{p}_\alpha(\tilde{y}|y) & =
\left\{ (\tilde{y}-Q\hat{\theta}_B)'
\Sigma_B^{-1}(\tilde{y}-Q\hat{\theta}_B)
+R+\frac{\|v_*\|^2}{\gamma}+s
\right\}^{-\frac{k+2a+2}{1-\alpha}} ,
\end{split}
\end{equation}
and where
\begin{equation}
\label{Sigma-U-etc}
\begin{split}
\Sigma_U &= \{2/(1-\alpha)\}I+QDQ'  \\
\hat{\theta}_B &=(C-I)(C+(1-\alpha)D/2)^{-1}v  \\ 
\Sigma_B &=
\{2/(1-\alpha)\}I+Q(C-I)D(C+\{(1-\alpha)/2\}D)^{-1}Q'  \\
R(v) &= v'(\{(1-\alpha)/2\}D+I)D^{-1}(C+\{(1-\alpha)/2\}D)^{-1}v 
\end{split}
\end{equation}
\end{thm}
\begin{proof}
 See Appendix.
\end{proof}
The first term $ \hat{p}_{\{U,\alpha\}}(\tilde{y}|y)$ is
the best invariant predictive density, and is Bayes with 
respect to the right invariant prior $\pi(\theta,\mu,\eta)=\eta^{-1}$.
Upon normalizing, $ \hat{p}_{\{U,\alpha\}}(\tilde{y}|y)$ is
multivariate-$t$ with the mean $Qv=\tilde{X}\hat{\beta}_U$.
We omit the straightforward calculation.
\cite{Liang-Barron-2004} show that $ \hat{p}_{\{U,\alpha\}}(\tilde{y}|y)$
has a constant minimax risk.

The second term, $\check{p}_\alpha$, is a pseudo multivariate-$t$
density with the mean vector $Q\hat{\theta}_B$. 
Since $\| \hat{\theta}_B\| \leq \|v\|$ is clearly satisfied, $\check{p}_\alpha$
induces a shrinkage effect toward the origin.
The complexity in the second term  is reduced considerably
with the choice $C=I$, in which case $\hat{\theta}_B=0$, 
$\Sigma_B=\{2/(1-\alpha)\}I$ and $R(v)=v'D^{-1}v$.
However, since the covariance matrix of $v$, $\sigma^2 D$,
is diagonal but not necessarily a multiple of $I$, the introduction 
of $C \neq I$ seems reasonable.
Indeed in the context of ridge regression, \cite{Casella-1980}
and \cite{Maru-Straw-2005}
have argued that shrinking unstable components more than
stable components is reasonable.
An ascending sequence of $c_i$'s leads to this end.
Hence the complexity, while perhaps not pleasing, is nevertheless 
potentially useful.

\subsection{Case ii: $ \alpha =1$}
\label{sec:case-2}
\begin{thm} \label{thm:case-2}
 The generalized Bayes predictive distribution under
$D_1$ divergence with respect to the prior \eqref{our-canonical-prior}
is  normal distribution $N_m( \hat{\theta}_{\nu,C}, \hat{\sigma}^2_{\nu,C} I_m) $
where
\begin{equation*}
\begin{split}
 \hat{\theta}_{\nu,C} &=\left( I - \frac{\nu}{\nu+1+W}C^{-1}\right)V \\
\hat{\sigma}^2_{\nu,C} &= \left(1- \frac{\nu}{\nu+1+W}\right)\frac{S}{n-k}
\end{split}
\end{equation*}
and where $W=\{V'C^{-1}D^{-1}V+\|V_*\|^2/\gamma\}/S$ and $\nu=(k+2a+2)/(n-k)$.
\end{thm}
\begin{proof}
See Appendix.
\end{proof}
It is quite interesting to note that
the Bayesian predictive density $\hat{p}_\alpha(\tilde{y}|y)$ for
$\alpha \in [-1,1)$ given in Section \ref{sec:case-1} converges to
$ \phi_m(\tilde{y}, Q \hat{\theta}_{\nu,C}, \hat{\sigma}^2_{\nu,C})$
 as $\alpha \to 1$
where  $\phi_m(\cdot, \xi, \tau^2)$ 
denotes the $m$-variate normal density with the mean vector
$\xi $ and
the covariance matrix $\tau^2 I_m$.

Since the Bayes solution is the plug-in predictive density as shown in 
Theorem \ref{thm:case-2}, 
we pay attention only to
the properties of plug-in predictive density under the loss $D_{1}$.
The $\alpha$-divergence with $\alpha=1$, 
from $\phi_m(\tilde{y}, Q\hat{\theta}, \hat{\sigma}^2)$,
the predictive normal density with plug-in
estimators $\hat{\theta}$ and $\hat{\sigma}^2$,
to $\phi_m(\tilde{y}, Q\theta, \sigma^2) $, the true normal density,
is given by
\begin{equation}
\label{loss-1}
\begin{split}
& \int \log \frac{\phi_m(\tilde{y}, Q\hat{\theta}, \hat{\sigma}^2)}
{\phi_m(\tilde{y}, Q\theta, \sigma^2)}\phi_m(\tilde{y},
 Q\hat{\theta}, \hat{\sigma}^2)d\tilde{y}  \\
& = \int \left\{ -\frac{m}{2} \log \frac{\hat{\sigma}^2}{\sigma^2}+
\frac{\| \tilde{y}- Q\theta \|^2}{2\sigma^2}
-\frac{\| \tilde{y}- Q\hat{\theta}\|^2}{2\hat{\sigma}^2} \right\}
\phi_m(\tilde{y}, Q\hat{\theta}, \hat{\sigma}^2)d\tilde{y}   \\
& = -\frac{m}{2} \log \frac{ \hat{\sigma}^2}{\sigma^2}
- \frac{m}{2} + \int \left\{ \frac{\| \tilde{y}- Q\hat{\theta} +Q\hat{\theta} 
-Q\theta \|^2}{2\sigma^2}\right\}
\phi_m(\tilde{y}, Q\hat{\theta}, \hat{\sigma}^2)d\tilde{y}   \\
&= \frac{\| \hat{\theta} -\theta \|^2}{2\sigma^2}+
\frac{m}{2} \left\{ \frac{\hat{\sigma}^2}{\sigma^2}- \log
\frac{\hat{\sigma}^2}{\sigma^2}-1\right\}  \\
&=\frac{1}{2}\left\{ L_1(\hat{\theta},\theta,\sigma^2)
+m  L_2(\hat{\sigma}^2,\sigma^2) \right\}.
\end{split}
\end{equation}
In \eqref{loss-1}, $L_1$ denotes
the scale invariant quadratic loss 
\begin{equation*}
 L_1(\hat{\theta},\theta,\sigma^2)
=\frac{(\hat{\theta}-\theta)'(\hat{\theta}-\theta)}{\sigma^2} 
\end{equation*}
for $\theta$ 
and $L_2$ denotes the Stein's or entropy loss
\begin{equation*}
  L_2(\hat{\sigma}^2,\sigma^2)=\frac{\hat{\sigma}^2}{\sigma^2}-\log
\frac{\hat{\sigma}^2}{\sigma^2}-1,
\end{equation*}
for $\sigma^2$.
Hence when the prediction problem under $\alpha$-divergence
with $\alpha = 1$
is considered from the Bayesian point of view,
the Bayesian solution is the normal distribution with plug-in 
Bayes estimators and
 the prediction problem reduces to the simultaneous
estimation problem of $\theta$ and $\sigma^2$ under the sum of losses
as in \eqref{loss-1}. 

%

\section{Improved minimax predictive densities under $D_1$}
\label{sec:minimax-D_1}
In this section, we give analytical results
on minimaxity under $D_1$ loss.
As pointed out in the previous section,
 the prediction problem under $D_1$ loss, reduces to the simultaneous
estimation problem of $\theta$ and $\sigma^2$ under the sum of losses
as in \eqref{loss-1}. 
Clearly the UMVU estimators of $\theta$ and $\sigma^2$ 
are $ \hat{\theta}_U=V  $ and $ \hat{\sigma}_U^2=S/(n-k)$.
These are also
generalized Bayes estimators with respect to
the the right invariant prior $ \pi(\theta, \mu, \eta)=\eta^{-1}$ and are hence minimax.
The constant minimax risk is
given by $\mbox{MR}_{\theta,\sigma^2}$ where
\begin{equation} 
\label{m-risk}
\mbox{MR}_{\theta,\sigma^2}= \frac{1}{2}\left\{ \mbox{tr}D
+m\left( \log \gamma - \frac{\Gamma'(\gamma)}{\Gamma(\gamma)}\right)\right\}
\end{equation}
and $\gamma=(n-k)/2$.

Recall that from observation $y$, 
there exist independent sufficient
statistics given by \eqref{canonical-observation}:
\begin{equation*} 
 V \sim N_l(\theta,\eta^{-1} D), \quad V_* \sim N_{k-l}(\mu,\eta^{-1} I),
 \quad \eta S \sim  \chi_{n-k}^2
\end{equation*}
where $\eta=\sigma^{-2}$, $l=\min(k,m)$, 
$D=\mbox{diag}(d_1,\dots,d_l)$ and $d_1 \geq \dots \geq d_l$. 
When $m \geq k$, $V_*$ is empty.

In the variance estimation problem of $\sigma^2$ under $L_2$,
\cite{Stein-1964} showed that $S/(n-k)$ is dominated by
\begin{equation}
\hat{\sigma}^2_{ST}= \min\left( \frac{S}{n-k}, \frac{V'D^{-1}V+S}{l+n-k}\right),
\end{equation}
for any combination of $\{ n,k,m \}$ including even $l=\min(k,m)=1$.
Hence, in the simultaneous estimation problem of $\theta$ and $\sigma^2$,
we easily see that $ \{\hat{\theta}_U, \hat{\sigma}_U^2\}$
is dominated by $\{\hat{\theta}_U, \hat{\sigma}^2_{ST}\}$ and hence
have the following result.
\begin{proposition}
The estimator $\{\hat{\theta}_U,\hat{\sigma}^2_U\}$ is inadmissible 
for any combination of $\{ n,k,m \}$.
\end{proposition}
The improved solution, $\{\hat{\theta}_U, \hat{\sigma}^2_{ST}\}$,
is unfortunately not Bayes.
When $l \geq 3$ and
\begin{equation} \label{assumption:d}
\textstyle{l-2 \leq 2\left(d_1^{-1}\sum_{i=1}^l d_i -2 \right)},
\end{equation}
we can construct a Bayesian solution using our earlier studies as follows.
In the estimation problem of $\theta$ under $L_1$, 
\cite{Maru-Straw-2005}
showed that the generalized Bayes estimator of $\theta$ with respect to 
the harmonic-type prior 
\begin{equation} \label{harmonic-1}
\pi_{S,1}(\theta, \eta)=\{\theta'D^{-1}\theta\}^{1-l/2}
\end{equation}
improves on the UMVU estimator $ \hat{\theta}_U$
when $l \geq 3$ and \eqref{assumption:d} is satisfied.
In the variance estimation problem of $\sigma^2$ under $L_2$,
although \cite{Maru-Straw-2006} did not state so explicitly, 
they showed that the generalized Bayes estimator of $\sigma^2$ with respect to 
the same prior \eqref{harmonic-1}
dominates the UMVU estimator $\hat{\sigma}_U^2$ when $l \geq 3$.
Hence the prior \eqref{harmonic-1} gives an improved Bayesian solution
in the simultaneous estimation problem of $\theta$ and $\sigma^2$
when $l \geq 3$ and \eqref{assumption:d} is satisfied.
(Note that under the special assumption \eqref{special-design-matrices} 
introduced in Section \ref{sec:intro}, 
$D$ becomes the multiple of identity matrix
and hence \eqref{assumption:d} is automatically satisfied.)

However, in the above construction of the Bayesian solution, 
two assumptions, $l \geq 3$ and \eqref{assumption:d}, are needed.
Further even if $m <k$ and $V_*$ exists, the Bayes procedure does not depend on $V_*$.
This is not desirable because the statistic $V_*$ has some information about
$\eta$ or $\sigma^2$. In fact, the Stein-type estimator of variance 
\begin{equation}
\hat{\sigma}^2_{ST*}=\min\left( \frac{S}{n-k}, \frac{\|V_*\|^2+S}{n-l}\right),
\end{equation}
as well as $\{\hat{\sigma}^2_{ST}\}$ dominates $\hat{\sigma}_U^2$ and hence
$\{\hat{\theta}_U, \hat{\sigma}^2_{ST*}\}$ also dominates 
$ \{\hat{\theta}_U, \hat{\sigma}_U^2\}$
in the simultaneous estimation problem.

Now we show that a subclass of the generalized Bayes procedure under $D_1$ 
given in Section \ref{sec:case-2} improves on the generalized Bayes procedure
with respect to the right invariant prior.
We assume neither $l \geq 3$ nor \eqref{assumption:d}.
Additionally the proposed procedure does depend on $V_*$ if it exists.
%

\begin{thm} \label{thm:minimaxity}
The generalized Bayes estimators of Theorem \ref{thm:case-2},
\begin{equation*}
\begin{split}
 \hat{\theta}_{\nu,C} &=\left( I - \frac{\nu}{\nu+1+W}C^{-1}\right)V \\
\hat{\sigma}^2_{\nu,C} &= \left(1- \frac{\nu}{\nu+1+W}\right)\frac{S}{n-k},
\end{split}
\end{equation*}
where $W=\{V'C^{-1}D^{-1}V+\|V_*\|^2/\gamma\}/S$,
 dominate the UMVU estimators ($V$ and $S/(n-k)$) under the loss
\eqref{loss-1} 
if $\gamma \geq 1$ and $0< \nu \leq \min(\nu_1,\nu_2,\nu_3)$ where
\begin{equation*}
\begin{split}
 \nu_1 &= 4\frac{\sum(d_i/c_i)-2\max(d_i/c_i) +m/(n-k)}
{2\max(d_i/c_i)(n-k+2) +m} \\
\nu_2 &= \frac{ 4\{\sum(d_i/c_i)-\max(d_i/c_i)\} + 2m/(n-k)}
{(n-k-2)\max(d_i/c_i)+m} \\
 \nu_3 &= \frac{4}{m}\sum\frac{d_i}{c_i}.
\end{split}
\end{equation*}
\end{thm}
\begin{proof}
 See Appendix.
\end{proof}
Clearly $\nu_2$ and $\nu_3 $ are always positive.
Now consider $\nu_1$. Assume $\nu_1$ is negative for fixed $C_0$.
But there exits $ g_0>1$ such that $C=g_0C_0$ makes $\nu_1$ positive.
Hence we can freely choose an ascending sequence of $c_i$'s which
guarantees the minimaxity of $ (\hat{\theta}_{\nu,C},\hat{\sigma}^2_{\nu,C}) $ 
and increased shrinkage of unstable components. 

\begin{remark}
We make some comments about domination results under the $D_1$ loss 
for the case of a known variance, say $\sigma^2=1$.
By \eqref{canonical-observation} and \eqref{loss-1}, 
the prediction problem under the $D_1$ loss
reduces to the problem of estimating an $l$-dimensional mean
vector $\theta$
under the quadratic loss $L_1(\hat{\theta},\theta)=\|\hat{\theta}-\theta\|^2$
in the case where there exists a sufficient statistic $V \sim N_l(\theta,D)$.
It is well known that 
the UMVU estimator $V$ is admissible when $l=1,2$ and inadmissible when
$l \geq 3$. Minimax admissible estimators for $l \geq 3$
have been proposed by many researchers including
\cite{Strawderman-1971}, \cite{Berger-1976}, \cite{Fourdrinier-etal-1998}, and
\cite{Maruyama-2003b}.
On the other hand, for KL (i.e.~$D_{-1}$) loss,
\cite{George-etal-2006} used some techniques including 
the heat equation and Stein's identity, and eventually
found a new identity 
which links KL risk reduction to Stein's unbiased estimate 
of risk reduction. 
Based on the link, they obtained sufficient conditions on the Bayesian predictive density 
for minimaxity.
Hence we expect that there should exist an analogous relationship between the
prediction problem under the $D_\alpha$ loss for $|\alpha|<1$ and
the problem of estimating the mean vector. 
As far as we know, this is still an open problem.
\end{remark}
%

\section{Concluding remarks} \label{sec:cr}
In this paper we have studied the construction 
and behavior of generalized Bayes predictive densities 
for normal linear models with unknown variance under $\alpha$-divergence loss. 
In particular we have shown that the best equivariant, 
(Bayes under the right invariant prior) and minimax predictive density under $D_1$
is inadmissible in all dimensions and for all residual degrees of freedom. 
We have found a class of improved hierarchical generalized Bayes procedures,
which gives a solution to Problem 2-1 of Section \ref{sec:intro}.

The domination results in this paper are closely related to those 
in \cite{Maru-Straw-2005, Maru-Straw-2006} for the respective problems 
of estimating the mean vector under quadratic loss 
and the variance under Stein's loss. 
In fact a key observation that aids in the current development 
is that the Bayes estimator under $D_1$ loss is a plug-in estimator normal density 
with mean vector and variance closely related to those of the above papers, 
and that the $D_1$ loss is the sum of a quadratic loss in the mean 
and Stein's loss for the variance.

We expect that an extension of a hierarchical prior given in Section \ref{sec:case-1},
for the prediction problem under the $D_{\alpha}$ loss for $-1 \leq \alpha <1$,
can form a basis to solve Problem 2-2 of Section \ref{sec:intro}.
To date, unfortunately, we have been less successful in extending the domination results 
to the full class of $\alpha$-divergence losses.

\appendix
\section{Appendix section}

\subsection{Proof of Theorem \ref{thm:case-1}}
The Bayesian predictive density $\hat{p}_\alpha(\tilde{y}|y)$
under the divergence $D_\alpha$ for
general $\alpha \in [-1,1)$ 
is proportional to
\begin{equation}
\left[
\iiint \{p(\tilde{y}|\theta,\eta)\}^{\frac{1-\alpha}{2}}p(v|\theta,\eta)
p(v_*|\mu,\eta)p(s|\eta)\pi(\theta,\mu,\eta)d\theta d\mu d\eta 
\right]^{\frac{2}{1-\alpha}},
\end{equation}
and hence the the integral in brackets is concretely written as
\begin{equation}
\label{integral-in-brackets}
\begin{split}
& \iiiint \eta^{\frac{m}{2}\frac{1-\alpha}{2}}
\exp\left(-\frac{\eta}{2}\frac{1-\alpha}{2}\| \tilde{y}-Q\theta\|^2\right) 
\eta^{\frac{n-k}{2}}\exp\left(-\frac{\eta s}{2}\right)  \\
& \quad \times \eta^{\frac{l}{2}}\exp\left( -\frac{\eta}{2}
(v-\theta)'D^{-1}(v-\theta)\right)
\eta^{\frac{k-l}{2}}\exp\left( -\frac{\eta}{2} \|v_*-\mu\|^2\right) 
 \\
& \quad \times
\frac{\lambda^{l/2}\eta^{l/2}}{\prod (c_i-\lambda)^{1/2}}
\exp\left( -\frac{\eta}{2}\theta' \left(D^{-1}+\frac{1-\alpha}{2}I_l
\right)\left(C/\lambda-I_l\right)^{-1}\theta\right)  \\
& \quad \times \left(\frac{\lambda \eta}{\gamma-\lambda}\right)^{(k-l)/2}
\exp\left( -\frac{\eta}{2} \frac{\lambda \|\mu\|^2}{\gamma-\lambda} \right)
\eta^a\lambda^a(1-\lambda)^b d\theta d\mu d\eta d\lambda. 
\end{split}
\end{equation}
To aid in the simplification of 
the integration with respect to $\theta$, we first
re-express those terms involving $\theta$ by completing the square,
and neglecting, for now, the factor $ \eta(1-\alpha)/4$.
Let $D_*=\{(1-\alpha)/2\}D$. Then
\begin{equation*}
\begin{split}
& \| \tilde{y} -Q\theta\|^2+ (v-\theta)'D_*^{-1}(v-\theta)+\theta'(I+D_*^{-1})
(C/\lambda-I)^{-1}\theta \\
&= \theta'(I+D_*^{-1})(I-C^{-1}\lambda)^{-1}\theta
-2\theta'(Q'\tilde{y}+D_*^{-1}v)+\|\tilde{y}\|^2+v'D_*^{-1}v \\
&= \{\theta - (I+D_*^{-1})^{-1}(I-C^{-1}\lambda)(Q'\tilde{y}+D_*^{-1}v)\}'
\{(I+D_*^{-1})(I-C^{-1}\lambda)^{-1}\} \\
& \qquad \times \{\theta -
 (I+D_*^{-1})^{-1}(I-C^{-1}\lambda)(Q'\tilde{y}+D_*^{-1}v)\} \\
& \quad - (Q'\tilde{y}+D_*^{-1}v)'\{(I+D_*^{-1})^{-1}(I-C^{-1}\lambda)\}
(Q'\tilde{y}+D_*^{-1}v)  +\|\tilde{y}\|^2+v'D_*^{-1}v .
\end{split}
\end{equation*}
The ``residual term'', 
\begin{equation*}
- (Q'\tilde{y}+D_*^{-1}v)'\{(I+D_*^{-1})^{-1}(I-C^{-1}\lambda)\}
 (Q'\tilde{y}+D_*^{-1}v)  +\|\tilde{y}\|^2+v'D_*^{-1}v ,
\end{equation*}
may be expressed as $ A+ \lambda\{B-A\}$
where
\begin{equation}
\label{AA}
\begin{split}
A &= A(\tilde{y},v,D_*,Q) \\
&= \|\tilde{y}\|^2+v'D_*^{-1}v
 -(Q'\tilde{y}+D_*^{-1}v)'(I+D_*^{-1})^{-1}
(Q'\tilde{y}+D_*^{-1}v)    \\
& = \{2/(1-\alpha)\}(\tilde{y}-Qv)'\Sigma_U^{-1}(\tilde{y}-Qv), 
\end{split}
\end{equation}
where $\Sigma_U $ is given by \eqref{Sigma-U-etc} and
\begin{equation}
\label{BB}
\begin{split}
B & = B(\tilde{y},v,C,D_*,Q) \\
&= \|\tilde{y}\|^2+v'D_*^{-1}v 
 - (Q'\tilde{y}+D_*^{-1}v)'(I+D_*^{-1})^{-1}(I-C^{-1})(Q'\tilde{y}+D_*^{-1}v)   \\
& =\{2/(1-\alpha)\} \left\{(\tilde{y}-Q\hat{\theta}_B)'\Sigma_B^{-1}
 (\tilde{y}-Q\hat{\theta}_B) \right\}  \\
& \qquad  +\{2/(1-\alpha)\}
\left\{v'(\{(1-\alpha)/2\}D+I)D^{-1}(C+\{(1-\alpha)/2\}D)^{-1}v\right\},
\end{split}
\end{equation}
where $\hat{\theta}_B$ and $\Sigma_B$ are given by \eqref{Sigma-U-etc}.
The third equalities in \eqref{AA} and \eqref{BB} 
will be proved in Lemma \ref{lemma-appendix} below.
Similarly we may re-express the terms involving $\mu$ as
\begin{equation*}
 \|v_*-\mu\|^2+ \frac{\lambda \|\mu\|^2 }{\gamma-\lambda}
 =\frac{\gamma}{\gamma-\lambda}\left\| \mu 
\left(1-\{1-\lambda/\gamma\}v_*\right)\right\|^2+ \lambda
\frac{\|v_*\|^2}{\gamma}.
\end{equation*}
After integration with respect to $\theta$ and $\mu$,
the integral given by \eqref{integral-in-brackets} is proportional to
\begin{equation}
\label{integral-lambda}
\begin{split}
& \iint \eta^{(1-\alpha)m/4+n/2+a}\lambda^{k/2+a}(1-\lambda)^{b}
 \\ 
& \qquad \exp\left(-\frac{\eta}{2}\left\{ \frac{1-\alpha}{2}A+s +\lambda
\left( \frac{1-\alpha}{2}(B-A)+\frac{\|v_*\|^2}{\gamma}
 \right)\right\}\right) d\eta d\lambda  \\
& \propto \int_0^1 \lambda^{k/2+a}(1-\lambda)^{(1-\alpha)m/4+(n-k)/2-1}
 \\ 
& \quad \left\{ \frac{1-\alpha}{2}A+s +\lambda
\left( \frac{1-\alpha}{2}(B-A)+\frac{\|v_*\|^2}{\gamma}
 \right)\right\}^{-(1-\alpha)m/4-n/2-a-1} d\lambda. 
\end{split}
\end{equation}
Note that in an identity given by 
\cite{Maru-Straw-2005}, (See page 1758)
\begin{equation} 
\label{annals-formula}
\begin{split}
& \int_{0}^{1}\lambda^{\alpha}(1-\lambda)^{\beta}(1+w \lambda)^{-\gamma}d
  \lambda \\
& \qquad = \frac{1}{(w+1)^{\alpha+1}}\int_{0}^{1}t^{\alpha}(1-t)^{\beta}
\left\{1-\frac{tw}{w+1} \right\}^{-\alpha-\beta+\gamma-2}dt,
\end{split}
\end{equation}
the integral of the right-hand side reduces the beta function
$Be(\alpha+1,\beta+1)$ when $ -\alpha-\beta+\gamma-2=0 $.
Hence the integral \eqref{integral-lambda} is exactly proportional to
\begin{equation} 
\label{integral-simple}
  \left( \frac{1-\alpha}{2}A+s \right)^{-(1-\alpha)m/4-(n-k)/2 } 
 \left( \frac{1-\alpha}{2}B+s +\frac{\|v_*\|^2}{\gamma}
 \right)^{-k/2-a-1}.
\end{equation}
Since the Bayesian predictive density $\hat{p}_\alpha(\tilde{y}|y)$
with respect to the prior $\pi(\theta,\mu,\eta)$ is 
proportional to the integral \eqref{integral-simple} to the $2/(1-\alpha)$
power, the theorem follows.

\begin{lemma} \label{lemma-appendix}
Let $F$ and $D_*$ be diagonal matrix. The matrix $Q$ is assumed to
 satisfy $Q'Q=I$. Then
\begin{equation*}
\begin{split}
& G(\tilde{y},v,F,D_*,Q)  \\
&= 
  \|\tilde{y}\|^2+v'D_*^{-1}v - 
(Q'\tilde{y}+D_*^{-1}v)'(I+D_*^{-1})^{-1}F
(Q'\tilde{y}+D_*^{-1}v) 
\end{split}
\end{equation*}
is re-expressed as
\begin{equation*}
\begin{split}
& \{\tilde{y}- QF(I+D_*(I-F))^{-1}v\}'\{I+QFD_*(I+D_*(I-F))^{-1}Q'\}^{-1} \\
& \quad \qquad \times \{\tilde{y}- QF(I+D_*(I-F))^{-1}v\} \\
& \qquad + v'(D_*+1)(I-F)D_*^{-1}(I+D_*(I-F))^{-1}v.
\end{split}
\end{equation*}
\end{lemma}
\begin{proof}
The function $ G(\tilde{y},v,F,D_*,Q)$ is re-expressed as
\begin{equation*}
\begin{split}
 G &= 
 \tilde{y}'(I-Q(I+D_*^{-1})^{-1}FQ')\tilde{y}- 2\tilde{y}'Q(I+D_*)^{-1}Fv \\
&\qquad +v'D_*^{-1}\{I-(I+D_*)^{-1}F\}v .
\end{split}
\end{equation*}
Since
\begin{equation}
 (I-Q(I+D_*^{-1})^{-1}FQ')^{-1}=
I+QFD_*(I+D_*(I-F))^{-1}Q',
\end{equation}
we obtain
\begin{equation*}
  \{I+QFD_*(I+D_*(I-F))^{-1}Q'\}Q(I+D_*)^{-1}F 
 = QF(I+D_*(I-F))^{-1}
\end{equation*}
and 
\begin{equation*}
\begin{split}
& F(I+D_*)^{-1}Q' \{I+QFD_*(I+D_*(I-F))^{-1}Q'\}Q(I+D_*)^{-1}F \\
& = F^2(I+D_*)^{-1}(I+D_*(I-F))^{-1}.
\end{split}
\end{equation*}
Hence we have
\begin{equation*}
\begin{split}
G &= \{\tilde{y}- QF(I+D_*(I-F))^{-1}v\}'(I+QFD_*(I+D_*(I-F))^{-1}Q')^{-1} \\
& \quad \qquad \times \{\tilde{y}- QF(I+D_*(I-F))^{-1}v\} \\
& \qquad  -v'F^2(I+D_*)^{-1}(I+D_*(I-F))^{-1}v  + v'D_*^{-1}\{I-(I+D_*)^{-1}F\}v.
\end{split}
\end{equation*}
Since the matrix for quadratic form of $v$ in the ``residual term''
may be written as
\begin{equation*}
\begin{split}
& D_*^{-1}\{I-(I+D_*)^{-1}F\} -F^2(I+D_*)^{-1}(I+D_*(I-F))^{-1} \\
& =(D_*+I)(I-F)D_*^{-1}(I+D_*(I-F))^{-1},
\end{split}
\end{equation*}
the lemma follows.
\end{proof}

\subsection{Proof of Theorem \ref{thm:case-2}}
The Bayes predictive density $\hat{p}_\alpha(\tilde{y}|y)$
under the divergence $D_\alpha$ for $\alpha=1$
is proportional to
\begin{equation}
\begin{split}
& \exp \left\{
\iiint \log p(\tilde{y}|\theta,\eta)p(v|\theta,\eta)
p(v_*|\mu,\eta)p(s|\eta)\pi(\theta,\mu,\eta)d\theta d\mu d\eta 
\right\} \\
&\propto \exp\left\{ \int \left(-\eta\frac{\|\tilde{y}-
 Q\theta\|^2}{2} \right)
\pi(\theta,\mu,\eta|v,v_*,s)d\theta d\mu d\eta \right\}  \\
&\propto \exp\left( -\frac{E(\eta|v,v_*,s)}{2}
\left\| \tilde{y} - Q
\frac{E[\eta\theta|v,v_*,s]}{E[\eta|v,v_*,s]}
\right\|^2 \right). 
\end{split}
\end{equation}
Hence the Bayes solution with respect to the prior
density $\pi(\theta,\mu,\sigma^2)$ under $D_1$
is the plug-in normal density
\begin{equation*}
\hat{p}_\alpha(\tilde{y}|y) =
\phi_m(\tilde{y}, Q\hat{\theta}_\pi, \hat{\sigma}_\pi^2)
\end{equation*}
where  $\phi_m(\cdot, Q\hat{\theta}_\pi, \hat{\sigma}_\pi^2)$ 
denotes the $m$-variate normal density with the mean vector
$Q\hat{\theta}_\pi $ and
the covariance matrix $\hat{\sigma}_\pi^2 I_m$ and where
$ \hat{\theta}_\pi$ and $ \hat{\sigma}^2_\pi$ are
given by
\begin{equation} 
\label{marginal-estimator}
\begin{split}
 \hat{\theta}_\pi & =\frac{E[\eta\theta|y]}
{E[\eta|y]}= v - \frac{D \nabla_v m(v,v_*,s)}
{2\{\partial/\partial s\}m(v,v_*,s)} , \\
\hat{\sigma}^2_\pi &= \frac{1}{E[\eta|y]} = -\frac{m(v,v_*,s)}
{2\{\partial/\partial s\}m(v,v_*,s)}, 
\end{split}
\end{equation}
and where $m(v,v_*,s)$ is the marginal density given by
\begin{equation}
 m(v,v_*,s)= \iiint p(v|\theta,\eta)p(v_*|\mu,\eta)p(s|\eta)
\pi(\theta,\mu,\eta)d\theta d\mu d\eta.
\end{equation}
Now we consider the marginal density of $(v,v_*,s)$ 
with respect to the prior $\pi(\theta,\mu,\eta)$,
\eqref{our-canonical-prior} with $\alpha=1$.
Using essentially the same calculations as in Section \ref{sec:case-1},
we obtain the  marginal density in the relatively simple form
\begin{equation}
 m(v,v_*,s) \propto s^{-(n-k)/2}
(v'C^{-1}D^{-1}v+\|v_*\|^2/\gamma+s)^{-(k/2+a+1)}.
\end{equation}
From the expression in \eqref{marginal-estimator},
a straightforward calculation gives the
 the estimators of $\theta$ and $\sigma^2$ in the simple form
\begin{equation}
\label{simple-form-estimator}
\begin{split}
 \hat{\theta}_{\nu,C} &=\left( I - \frac{\nu}{\nu+1+W}C^{-1}\right)V, \\
\hat{\sigma}^2_{\nu,C} &= \left(1- \frac{\nu}{\nu+1+W}\right)\frac{S}{n-k},
\end{split}
\end{equation}
where $W=\{V'C^{-1}D^{-1}V+\|V_*\|^2/\gamma\}/S$, respectively. 
This completes the proof.

\subsection{Proof of Theorem \ref{thm:minimaxity}}
\cite{Maru-Straw-2005} showed that, under the $L_1 $ loss,
the risk function of a general shrinkage estimator 
\begin{equation*}
\hat{\theta}_\phi=\left(I - \frac{\phi(W)}{W} C^{-1}\right)V
\end{equation*}
with suitable $\phi$ is given by
\begin{equation*}
\begin{split}
 &  E\left[ L_1(\hat{\theta}_\phi,\theta,\sigma^2)\right]
=E\left[ \frac{\| \hat{\theta}_\phi - \theta\|^2}{\sigma^2}\right] \\
& =  E\left[ \frac{\| V - \theta\|^2}{\sigma^2}\right]+
E \left[ \frac{\phi(W)}{W} \left\{ 
\psi(V,V_*,C,D,\nu) \bigg( (n-k+2)\phi(W)  \right. \right. \\
& \quad   \left. \left.
+4 \left\{ 1- \frac{W\phi'(W)}{\phi(W)}(1+\phi(W) \right\} 
 \bigg) -2\sum_{i=1}^l \frac{d_i}{c_i} \right\} \right],
\end{split}
\end{equation*}
where $\psi(v,v_*,C,D,\nu)$ is given by
\begin{equation*}
\psi(v,v_*,C,D,\nu)= \frac{v'C^{-2}v}{v'C^{-1}D^{-1}v+\|v_*\|^2/\gamma}.
\end{equation*}
For $\phi_\nu(w)=\nu w/(\nu+1+w)$, we have
\begin{equation*}
\begin{split}
&(n-k+2)\phi(w)  
+4 \left\{ 1- \frac{w\phi'(w)}{\phi(w)}(1+\phi(w) \right\} \\
&=  \frac{\{(n-k+2)\nu+4\}w^2+
 (\nu+1)\{\nu(n-k-2)+4\}w}{(1+\nu+w)^2 }
\end{split}
\end{equation*}
which is always positive when $n-k-2 \geq 0$.
Since $ \psi $ is bounded from above by $\max_{1 \leq i \leq l }
 d_i/c_i$, 
the risk function of $\hat{\theta}_\nu$ satisfies
\begin{equation*}
\begin{split}
&  E\left[L_1( \hat{\theta}_{\nu,C}, \theta, \sigma^2)\right]\\
& \leq \mbox{MR}_{\theta} +
E \left[ \frac{\nu}{1+\nu+W} \left\{ -2\sum \frac{d_i}{c_i} + 
\max \frac{d_i}{c_i}  \frac{\{(n-k+2)\nu+4\}W^2}{(1+\nu+W)^2}
 \right. \right. \\
& \quad \left. \left. + \max \frac{d_i}{c_i}
\frac{(\nu+1)\{\nu(n-k-2)+4\}W}{(1+\nu+W)^2} \right\} \right],
\end{split}
\end{equation*}
where $\mbox{MR}_{\theta}=\mbox{tr}D$.

Next we consider the risk function of 
$\hat{\sigma}^2_\phi=(1-\phi(W)/W)S/(n-k)$
where $ 0 < \phi(w)/w < 1$, which is given by
\begin{equation*}
\begin{split}
 E[L_2(\hat{\sigma}^2_\nu, \sigma^2)] 
&= E\left[ \left(1-\frac{\phi(W)}{W}\right)\frac{S}{(n-k)\sigma^2}
     -\log \frac{S}{(n-k)\sigma^2} \right.\\
& \quad \left.
-\log\left(1-\frac{\phi(W)}{W}\right) 
-1
    \right] \\
&= \mbox{MR}_{\sigma^2}
+E\left[ -\frac{\phi(W)}{W\sigma^2}\frac{S}{n-k}-\log
\left(1-\frac{\phi(W)}{W}\right)\right],
\end{split}
\end{equation*}
where $ \mbox{MR}_{\sigma^2}=  \log \gamma - \Gamma'(\gamma)/\Gamma(\gamma)$
and $\gamma=(n-k)/2$.
By the chi-square identity (See e.g.~\cite{Efron-Morris-1976}),
we have
\begin{equation*}
 E\left[ \frac{\phi(W)S}{W\sigma^2}\right]=
E\left[ (n-k+2)\frac{\phi(W)}{W}-2\phi'(W)\right].
\end{equation*}
Also using the relation
\begin{equation*}
-\log(1-x)=\sum_{i=1}^{\infty}\frac{x^{i}}{i} \leq x 
+\frac{1}{2}\frac{x^{2}}{1-x},
\end{equation*}
for $0 < x<1$, we have
\begin{equation*}
\begin{split}
& E[L_2(\hat{\sigma}^2_\nu, \sigma^2)]  \\
&\leq \mbox{MR}_{\sigma^2}
+E\left[ \frac{\phi(W)}{W}
\left\{ \frac{2}{n-k}\left( \frac{W\phi'(W)}{\phi(W)} -1 \right) 
+ \frac{1}{2}\max\frac{\phi(w)/w}{1-\phi(w)/w}
\right\} \right].
\end{split}
\end{equation*}
For $\phi(w)=\nu w/(\nu+1+w)$, we have
\begin{equation*}
\begin{split}
& E[L_2(\hat{\sigma}^2_{\nu,C}, \sigma^2)] \\
&\leq \mbox{MR}_{\sigma^2}
+E\left[ \frac{\nu}{1+\nu+W}
\left\{ -\frac{2}{n-k}\frac{W}{1+\nu+W} 
+ \frac{\nu}{2}
\right\} \right].
\end{split}
 \end{equation*}
Hence
\begin{equation*}
\frac{1}{2}E\left[L_1( \hat{\theta}_{\nu,C}, \theta, \sigma^2)\right]
+ \frac{m}{2} E[L_2(\hat{\sigma}^2_{\nu,C}, \sigma^2)]
 \leq \mbox{MR}_{\theta,\sigma^2}
- \frac{\nu}{2} E\left[ \frac{\psi(W)}{(1+\nu+W)^3}\right]
\end{equation*}
where $\mbox{MR}_{\theta,\sigma^2}$ is the minimax risk given by
\eqref{m-risk} and 
\begin{equation*}
\begin{split}
&  \psi(w) \\
& = \frac{w^2}{2}\left( 4\left\{ \sum\frac{d_i}{c_i}
-2\max \frac{d_i}{c_i} +\frac{m}{n-k}\right\} 
-\nu\left\{ 2\max \frac{d_i}{c_i}(n-k+2) +m \right\}\right) \\
& \quad +(\nu+1)w \left( 4\left\{ \sum\frac{d_i}{c_i}
-\max \frac{d_i}{c_i} \right\}+ \frac{2m}{n-k} 
-\nu \left\{ (n-k-2)\max \frac{d_i}{c_i}+m\right\}\right) \\
& \quad + \frac{(1+\nu)^2}{2}\left( 4 \sum\frac{d_i}{c_i} - \nu m
 \right) \geq 0.
\end{split}
\end{equation*}
Hence the theorem follows.

\end{document}